\documentclass[a4paper,oneside,11 pt]{amsart}
 
\usepackage[T1]{fontenc}
\usepackage[utf8]{inputenc}
\usepackage{amsthm}
\usepackage{amsmath}
\usepackage{amssymb}
\usepackage{mathrsfs}
\usepackage{graphicx}
\usepackage{enumitem} 
\usepackage{tikz}
\usepackage{natbib}
\usepackage{kantlipsum} 
\usepackage{bm}

\calclayout

\author{Léonard Cadilhac}

\newtheorem{definition}{Definition}[section]
\newtheorem{theorem}[definition]{Theorem}
\newtheorem{conjecture}[definition]{Conjecture}
\newtheorem{property}[definition]{Proposition}
\theoremstyle{definition}
\newtheorem{remark}[definition]{Remark}

\newtheorem{lemma}[definition]{Lemma}

\newtheorem{corollary}[definition]{Corollary}

\newcommand{\Zb}{\mathbb{Z}}
\newcommand{\Rb}{\mathbb{R}}
\newcommand{\Mb}{\mathbb{M}}

\newcommand{\Nb}{\mathbb{N}}

\newcommand{\Cb}{\mathbb{C}}
\newcommand{\Eb}{\mathbb{E}}
\newcommand{\Db}{\mathbb{D}}

\newcommand{\A}{\mathcal{A}}
\newcommand{\B}{\mathcal{B}}

\newcommand{\M}{\mathcal{M}}

\newcommand{\E}{\mathcal{E}}

\newcommand{\R}{\mathcal{R}}

\newcommand\Norm[1]{\left\| #1 \right\|}
\newcommand\md[1]{\left| #1 \right|}

\newcommand\Floor[1]{\left \lfloor {#1} \right \rfloor}

\newcommand\p[1]{\left( #1 \right)}
\newcommand\set[1]{\left\lbrace #1 \right\rbrace}
\newcommand{\normop}[1]{{\left\vert\kern-0.25ex\left\vert\kern-0.25ex\left\vert #1 
    \right\vert\kern-0.25ex\right\vert\kern-0.25ex\right\vert}}

\newcommand{\les}{\lesssim}

\newcommand{\e}{\varepsilon}

\newcommand{\Ind}{\mathbf{1}}

\newcommand\Sum[2]{\sum\limits_{#1}^{#2}}
\newcommand\Int[2]{\int_{#1}^{#2}}

\newcommand\addtag{\refstepcounter{equation}\tag{\theequation}}

\newcommand\restr[2]{{
  \left.\kern-\nulldelimiterspace 
  #1 
  \right|_{#2} 
  }}

\title{Majorization, interpolation and noncommutative Khintchine inequality}

\begin{document}

\maketitle

\begin{abstract}
Let $0<p<q\leq\infty$ and $\alpha \in (0,\infty]$. We give a characterization of quasi-Banach interpolation spaces for the couple \mbox{$(L_p(0,\alpha),L_q(0,\alpha))$} in terms of two monotonicity properties, extending known results which mainly dealt with Banach spaces. This enables us to recover recent results of Cwikel and Nilsson on sequence spaces and to solve a conjecture of Levitina, Sukochev and Zanin in the setting of function spaces. We apply the results obtained to characterize symmetric spaces in which the standard forms of the noncommutative Khintchine inequalities hold.
\end{abstract}

\section{Introduction}

This paper is motivated by two different problems, one in the classical theory of $L_p$-spaces and in particular symmetric spaces and the other in noncommutative harmonic analysis. The two have become closely related during the last decade, since noncommutative symmetric spaces have been found to be a nice setting in which to generalize classical theorems and formulate new results. Examples linked to the present paper can be found in \cite{LusXu07}, \cite{KalSuk08}, \cite{LeMSuk08}, \cite{DPPS11}, \cite{DirRic13} and \cite{LevSukZan18}. 

\subsection*{Question 1: Interpolation of $L_p$-spaces and right-majorization}

Characterisations and sufficient conditions garanteeing that a symmetric space is an interpolation space for a couple of $L_p$-spaces have been investigated in the past decades. Sufficient conditions can be formulated in terms of convexity, concavity or Boyd indices and characterisations rely on the computation of the $K$-functional for the couple $(L_p,L_q)$, see \cite{KalMon03} for a survey. In \cite{LevSukZan18}, Levitina, Sukochev and Zanin conjecture a new characterisation of sequence spaces $E$ for which there exists $p<2$ such that $E$ is an interpolation space for the couple $(\ell^p,\ell^2)$. It can be stated in terms of \emph{right-majorization} and \emph{right-monotonicity}. Let us define these notions as well as their left counterparts. For functions (or sequences) $f$ and $g$ which admit nonincreasing rearrangements $f^*$ and $g^*$, 
\begin{itemize}
\item we say that $f$ right-majorizes $g$ and write $f\triangleright g$ if: $\forall t>0, \Int{t}{\infty} f^* \geq \Int{t}{\infty} g^*,$
\item we say that $f$ left-majorizes $g$ and write $f\succ g$ if: $\forall t>0, \Int{0}{t} f^* \geq \Int{0}{t} g^*.$
\end{itemize}
For any $p,q \in (0,\infty)$, a symmetric quasi-Banach space $E$ is said to be:
\begin{itemize}
\item \emph{right-$q$-monotone} if there exists $C>0$ such that for all $f\in E$ and $g\in L_0$, 
$\md{f}^q \triangleright \md{g}^q \Rightarrow g\in E, \Norm{g}_E \leq C\Norm{f}_E.$
\item \emph{left-$p$-monotone} if there exists $C>0$ such that for all $f\in E$ and $g\in L_0$, 
$\md{f}^p \succ \md{g}^p \Rightarrow g\in E, \Norm{g}_E \leq C\Norm{f}_E.$
\end{itemize}
For a definition of $L_0$ see Subsection \ref{sub:sym}.
\begin{conjecture}[Levitina, Sukochev, Zanin]\label{conj:LevSukZan}
Let $E$ be a quasi-Banach symmetric sequence space. Then $E$ is right-$2$-monotone if and only if there exists $p\in (0,2]$ such that $E$ is an interpolation space between $\ell^p$ and $\ell^2$.
\end{conjecture}
Cwikel and Nilsson proved the conjecture for Banach spaces with the Fatou property in \cite{CwiNil18}. We recover their result and prove that the conjecture holds for function spaces.
\begin{theorem}\label{thm:intro:conjfonc}
Let $E$ be a quasi-Banach symmetric function space and $q\in (0,\infty)$. The two following properties are equivalent:
\begin{enumerate}
\item $E$ is right-$q$-monotone,
\item there exists $p\leq q$ such that $E$ is an interpolation space for the couple $(L_p,L_q)$.
\end{enumerate}
\end{theorem}
The first ingredient to prove this theorem is a characterisation of interpolation spaces for the couple $(L_p,L_q)$ given in Theorem \ref{thm:functionspace} \emph{i.e.} a symmetric space $E$ is an interpolation space for the couple $(L_p,L_q)$ if and only if it is left-$p$-monotone and right-$q$-monotone. Although we formulate the result differently, it is an generalisation to the quasi-Banach case of a known fact (see Theorem 7.2 in \cite{KalMon03}). The second is to remark that for any quasi-Banach symmetric space $E$, there exists $p>0$ such that $E$ is an interpolation space for the couple $(L_p,L_\infty)$. This can be obtained directly from the literature by considering the Boyd indices of $E$, $\alpha_E$ and $\beta_E$ (see Subsection \ref{subsection Boyd}).

In section \ref{3}, we extend part of the Lorentz-Shimogaki theorem to the quasi-Banach setting. This was, to the best of our knowledge, not a direct consequence of previous works despite the fact that very similar results can be found, see \cite{Cwi81}, \cite{Spa78}.

In section \ref{4}, the characterisation, which is our main theorem, is proven as well as its partial extension to sequence spaces. We also give an application to $p$-convexifications. 

In section \ref{5}, we relate left-$p$-monotonicity to $p$-convexity and right-$q$-monotonicity to $q$-concavity. A direct consequence of this, combined with the main theorem is to recover and generalize known sufficient conditions for a symmetric space $E$ to be an interpolation space for the couple $(L_p,L_q)$. In the following theorem, we compile some consequences of Theorem \ref{thm:LorentzShimogakiquasiBanach}, Theorem \ref{thm:functionspace}, Section \ref{5} and Subsection \ref{subsection Boyd} in that direction.
\begin{theorem}\label{thm : compilation}
Assume that $E$ is a symmetric space and $p,q\in (0,\infty)$ with $p<q$. For $E$ to be an interpolation space for the couple $(L_p,L_q)$, it suffices that $E$ verify one of the conditions among:
\begin{itemize}
\item[P1.] $E$ is left-$p$-monotone,
\item[P2.] $p < 1/\beta_E$,
\item[P3.] $E$ has the Fatou property and is $p$-convex, 
\item[P4.] $E$ is an interpolation space for the couple $(L_p,L_\infty)$,
\end{itemize}
{\rm and} one of the conditions among:
\begin{itemize}
\item[Q1.] $E$ is right-$q$-monotone,
\item[Q2.] $q > 1/\alpha_E$,
\item[Q3.] $E$ has the Fatou property and is $q$-concave.
\end{itemize}
\end{theorem}
Theorems of this form have already appeared in the literature, for example in \cite[Theorem 3.2]{DPPS11}, \cite[Theorem 1]{AM04} or in the survey \cite{KalMon03}. The version presented here is quite general, except for the fact that we do not consider the hypothesis of being {\it separable} which could be used instead of the {\it Fatou property}.

\subsection*{Question 2: Noncommutative Khintchine inequalities in symmetric spaces} Noncommutative Khintchine inequalities have been introduced in \cite{Lus86} for $L_p$-spaces and have since been a crucial tool, in particular for the development of noncommutative harmonic analysis. They have been further studied by many different authors and in the general context of symmetric spaces (\cite{LusPis91}, \cite{LusXu07}, \cite{LeMSuk08}, \cite{DPPS11}). The aim of Khintchine inequalities is, given a specific sequence of random variables $(\xi_i)_{i\in\Nb}$ in a (noncommutative) probability space $\A$ (independent Rademacher variables, free Haar unitaries \textit{i.e.} freely independent variables which are uniformly distributed on the unit circle of $\Cb$ ...) and a measure space of coefficients $\M$ to provide computable expressions for norms of elements of the form:
$$Gx := \sum_{i=1}^\infty x_i \otimes \xi_i,$$
defined in $\M \overline{\otimes} \A$, where $x = (x_i)_{i\in\Nb}$ is a finite sequence of elements of $\M$. 

Before going into more details, let us introduce a notation that will be used in the remainder of the text. For quantities $A$ and $B$, we write $A(x) \les B(y)$ if there exists a constant $c$ independent of $x$ and $y$ such that $A(x) \leq cB(y)$.  Additionally, we write $A(x) \approx B(y)$ if $A(x)\les B(y)$ and $B(y)\les A(x)$.

The original Khintchine inequalities considered Rademacher variables and $\M$ commutative. In this case for any $p\in (0,\infty)$,
$$\Norm{Gx}_p \approx \Norm{\p{\sum_{i=0}^\infty \md{x_i}^2}^{1/2}}_p.$$
In the noncommutative context however, the formulation of noncommutative Khintchine inequalities in $L_p$ depends of whether $p\leq 2$ or $p\geq 2$. This is due to the fact that two different square functions can be defined:
$$S_c(x) = \p{\sum_{i=1}^\infty x_i^*x_i}^{1/2}\ \text{and}\ S_r(x) = \p{\sum_{i=1}^\infty x_ix_i^*}^{1/2}.$$
With these notations, the noncommutative Khintchine inequalities state that if $p\in (2,\infty)$:
$$
\Norm{Gx}_p \approx \max \p{\Norm{S_c(x)}_p, \Norm{S_r(x)}_p},
$$
and if $p\leq 2$:
$$
\Norm{Gx}_p \approx \inf \left\lbrace \Norm{S_c(z)}_p + \Norm{S_r(y)}_p:y+z = x \right\rbrace.
$$
It is therefore natural to try and characterize the symmetric spaces in which one of these (quasi-)norm equivalences hold. More precisely, denote by $S(\M)$ the space of finite sequences of finitely supported elements of $\M$ and define the following properties:
\begin{itemize}
\item $Kh_\cap(E,\M)$: for any $x\in S(\M)$, 
$$\Norm{Gx}_{E} \approx \max \p{\Norm{S_c(x)}_E, \Norm{S_r(x)}_E} =: \Norm{x}_{R_E\cap C_E},$$
\item $Kh_\Sigma(E,\M)$: for any $x\in S(\M)$, 
$$\Norm{Gx}_{E} \approx \inf \left\lbrace \Norm{S_c(z)}_E + \Norm{S_r(y)}_E:y+z = x \right\rbrace =: \Norm{x}_{R_E + C_E}.$$
\end{itemize}
If $\B(\ell^2) \overline{\otimes} L_{\infty}(0,1)$ embeds (by a unital trace preserving homomorphism) in $\M$, we characterize the symmetric spaces having properties $Kh_\cap(E,\M)$ and $Kh_\Sigma(E,\M)$ in terms of monotonicity and interpolation properties if the sequence $(\xi_i)_{i\in\Nb}$ is constituted of free Haar unitaries or independent Rademacher variables. 
\begin{theorem}\label{thm:intro:khin}
Let $E$ be a quasi-Banach symmetric space with the Fatou property. Assume that $\M = \B(\ell^2) \overline{\otimes} L_{\infty}(0,1)$ and that $(\xi_i)_{i\in\Nb}$ is a sequence of free Haar unitaries in $\A$. Denote by $\E: \B(\ell^2) \to \B(\ell^2)$ the conditional expectation onto the diagonal. Then the following properties are equivalent:
\begin{enumerate}[label=(\roman*)]
\item $Kh_\cap(E,\M)$,
\item $E$ is left-$2$-monotone,
\item $E$ is an interpolation space for the couple $(L_2,L_\infty)$,
\item $\forall x\in\M^+, \Norm{(\E\otimes Id)(x^2)^{1/2}}_E \les \Norm{x}_E$.
\end{enumerate}
The following properties are also equivalent:
\begin{enumerate}[label=(\roman*)]
\item $Kh_\Sigma(E,\M)$,
\item $E$ is right-$2$-monotone,
\item there exists $p<2$ such that $E$ is an interpolation space for the couple $(L_p,L_2)$,
\item $\forall x\in\M^+, \Norm{x}_E \les \Norm{(\E\otimes Id)(x^2)^{1/2}}_E$.
\end{enumerate}
\end{theorem}
A similar statement holds for Rademacher variables. In this case $Kh_\cap(L_\infty,\M)$ never holds but this problem can be dealt with thanks to works of Astashkin \citep{Ast09}. We are very grateful to D. Zanin for pointing out this reference to us. 

In the theorem above, the most difficult implications to prove are $(ii) \Rightarrow (i)$ (or $(iii) \Rightarrow (i)$) and for these we will refer to \cite{DirRic13} where the case of $Kh_\cap$ is handled, \cite{PisRic17} where $Kh_\Sigma$ is proven if $E$ is an $L_p$-space and \cite{Cad18} which allows to interpolate the previous result and obtain $Kh_\Sigma$ for a general symmetric space $E$. The equivalence between $(ii)$ and $(iii)$ is of purely commutative nature. It is where the two problems we are interested in intersect and is obtained in sections \ref{3} and \ref{4}. 

In section \ref{section:khintchine}, we prove that $(i) \Rightarrow (iv) \Rightarrow (ii)$. In particular, $(iv) \Rightarrow (ii)$ is an application of the following well-known theorem (see \cite{Hor54}):
\begin{theorem}[Schur-Horn]\label{thm:schurhorn}
Let $N\in\Nb$. Let $a,b \in \Rb_+^N$ be non-increasing sequences such that:
$$\forall k\leq N, \sum_{i=1}^k a_i \geq \sum_{i=1}^k b_i\ \text{and}\ \sum_{i=1}^N a_i = \sum_{i=1}^N b_i.$$
Then,
\begin{itemize}
\item $b$ belongs to the convex hull of $\set{(a_{\sigma(1)},\dots,a_{\sigma(N)}): \sigma \in S_n} \subset \Rb^N$,
\item there exists a Hermitian matrix in $M \in \Mb_n(\Cb)$ such that the eigenvalues of $M$ are given by $a$ and the diagonal of $M$ is given by $b$.
\end{itemize}
\end{theorem}

\section{Preliminaries}

\subsection{Interpolation}

For a detailed exposition of interpolation theory, see \cite{BerLof76}. We simply recall here the main definitions and properties that will be used later on. Note that interpolation theory is often defined in the context of Banach spaces but translates well to the quasi-Banach setting (\cite{BerLof76}, section 2.9). We start with the definition of an interpolation space.

\begin{definition}
Let $(A,B)$ be a compatible couple of quasi-Banach spaces. We say that a quasi-Banach space $E$ is an interpolation space for this couple if $A\cap B \subset E \subset A+B$ with constant $C>0$ if for every bounded operator $T:A+B \to A+B$ such that $\restr{T}{A}$ (resp.$\restr{T}{B}$) is a contraction from $A$ to $A$ (resp. $B$ to $B$), $T$ is bounded from $E$ to $E$ with norm less than $C$. If $C = 1$, we say that $E$ is an exact interpolation space.
\end{definition}

Both for explicit constructions of interpolation spaces with the real method and for the general theory of interpolation, the $K$-functional is a fundamental tool. It is defined as follows:

\begin{definition}
Let $(A,B)$ be a compatible couple of quasi-Banach spaces and $x\in A+B$. For all $t>0$ define the $K$-functional of $x$ by:
$$K_t(x,A,B) = \inf \{\Norm{y}_A + t\Norm{z}_B : y\in A, z\in B, y+z = x \}.$$
\end{definition}

It enables to state a simple sufficient condition for a space to be an interpolation space. 

\begin{definition}
Let $(A,B)$ be a compatible couple of quasi-Banach spaces. We say that a quasi-Banach space $E$ is $K$-monotone for $(A,B)$ with constant $C>0$ if $A \cap B \subset E \subset A+B$ and for all $x\in E$ and $y\in A+B$, verifying
$$\forall t>0, K_t(x,A,B) \geq K_t(y,A,B),$$
then $y\in E$ and 
$$\Norm{y}_E \leq C \Norm{x}_E.$$
\end{definition}

The following fact is well-known. We prove it here in the context of quasi-Banach spaces for completion.

\begin{property} \label{prop:KmonToInter}
If $E$ is $K$-monotone with constant $C$ for the couple $(A,B)$ then $E$ is an interpolation space between $A$ and $B$ with the same constant.
\end{property}

\begin{proof}
Let $T: A+B \to A+B$ be a bounded operator such that its restriction to $A$ (resp. $B$) is a bounded operator with norm $1$ on $A$ (resp. $B$). Let $x\in E$ and let $y = Tx$, $y\in A+B$. Let $t>0$, $A+tB$ is an exact interpolation space for the couple $(A,B)$ so $\Norm{x}_{A+tB} \geq \Norm{y}_{A+tB}$. This means that for all $t>0$, $K_t(x,A,B) \geq K_t(y,A,B)$. Hence, by $K$-monotonicity of $E$, $y\in E$ and $\Norm{y}_E \leq C\Norm{x}_E$. So $T$ defines a bounded operator of norm less than $C$ on $E$.
\end{proof}

If reciprocally, every interpolation space of a couple $(A,B)$ is $K$-monotone then we say that $(A,B)$ is an \emph{Calder\'on couple}.

\subsection{Symmetric spaces}\label{sub:sym}

We start by introducing some notations. Measures will be, if not mentioned otherwise, denoted by $\nu$. If $(\Omega,\nu)$ is a measure space, denote by $L_0(\Omega,\nu)$ or simply $L_0(\Omega)$ (if no confusion can occur) the set of measurable functions $f$ on $\Omega$ such that $\nu\p{\set{\md{f}>t}}$ is finite for some $t\in\Rb$. To any $f\in L_0(\Omega)$, we associate a non-increasing rearrangement which is a function in $L_0(0,\nu(\Omega))$, denoted by $f^*$ or $\mu(f)$ and defined for $t\in (0,\nu(\Omega))$ by:
$$f^*(t) = \mu_t(f) = \inf \set{\Norm{f\Ind_A}_\infty : \nu(\Omega \backslash A) \geq t}.$$
A \emph{quasi-Banach symmetric space} on $\Omega$ is a nonzero subspace of $L_0(\Omega)$ which is rearrangement invariant (the quasi-norm of a function only depends on its distribution) and equipped with an increasing quasi-norm.  A symmetric space $E$ is said to have the {\it Fatou property} if for every increasing net $(f_n)_{n\in I}$ of elements of $E$ such that $(\Norm{f_n}_E)_{n\in I}$ is bounded and $f_n \uparrow f$ a.e., then $f\in E$ and $\Norm{f_n}_E \uparrow \Norm{f}_E$. An introduction to symmetric spaces can be found in \cite{KrePetSem82}.

For the remainder of the paper, fix $\alpha \in (0,\infty]$. Denote simply by $L_p$, $p\in (0,\infty]$ the space $L_p(0,\alpha)$ where $(0,\alpha)$ is equiped with the Lebesgue measure. A quasi-Banach symmetric space on $(0,\alpha)$ will be called a symmetric function space. Similarly, a symmetric sequence space is a quasi-Banach symmetric space on $\Nb$ endowed with the counting measure. Note that if $\nu(\Omega) \leq \alpha$, we can define a space $E(\Omega)$ by:
$$E(\Omega) = \set{f\in L_0(\Omega) : f^* \in E}\ \text{and}\ \forall f\in E(\Omega), \Norm{f}_{E(\Omega)} = \Norm{f^*}_E.$$

We end this subsection by a technical lemma that will be useful later on.

\begin{lemma}[H\"older type inequality for positive operators] \label{lem:holder}
Let $T$ be a positive operator on $L_0$. Then, for any $a,b\in L_0^+$ and $s,s'>0$ such that $1 = s^{-1} + s'^{-1}$, $$T(ab) \leq T(a^s)^{1/s}T(b^{s'})^{1/s'}.$$
\end{lemma}

\begin{proof}
Recall that for any $x,y\in \Rb^+$,
$$xy = \inf_{\lambda>0} \frac{\lambda^sx^s}{s} + \frac{\lambda^{-s'}y^{s'}}{s'}.$$
Hence,
$$T(a^s)^{1/s}T(b^{s'})^{1/s'} = \inf_{\lambda>0} \frac{\lambda^sT(a^s)}{s} + \frac{\lambda^{-s'}T(b^{s'})}{s'} = 
\inf_{\lambda>0} T\p{\frac{\lambda^sa^s}{s} + \frac{\lambda^{-s'}b^{s'}}{s'}} \geq T(ab).$$
\end{proof}

\subsection{Interpolation of $L_p$-spaces}

The notion of left-$p$-monotonicity is almost equivalent to $K$-monotonicity for the couple $(L_p,L_\infty)$. Indeed, the $K$-functional of the couple $(L_p,L_{\infty})$ takes the following form (see \cite{Hol70}):
\[
K_t(x,L_p,L_{\infty}) \approx \p{\Int{0}{t^p} (x^*)^p}^{1/p}.\addtag\label{formula:KtLpLinfty}
\]
More generally, the $K$-functional of the couple $(L_p,L_q)$, $0<p<q<\infty$ has been computed up to constants depending on $p$ and $q$ and can be found in \cite{Hol70}. Let $r = (p^{-1} - q^{-1})^{-1}$. Then, for all $f$ in $L_p + L_q$ and $t>0$,
\[
K_t(f,L_p,L_q) \approx \p{\Int{0}{t^{r}} (f^*)^p}^{1/p} + t \p{\Int{t^{r}}{\infty} (f^*)^q}^{1/q}.\addtag \label{formula:KfuncLpLq}
\]
Combined with the following theorem, this allows, in many cases, to give a precise description of the interpolation spaces for the couple $(L_p,L_q)$ (\cite{LorShi71},\cite{Cwi81},\cite{Spa78}).

\begin{theorem}\label{thm:Kmon for Lp literature} 
Let $\alpha\in (0,\infty]$. The following couples of quasi-Banach spaces are Calder\'on couples:
\begin{itemize}
\item (Lorentz-Shimogaki) $p\geq 1$, $(L_p(0,\alpha),L_{\infty}(0,\alpha))$ and $(L_1(0,\alpha),L_p(0,\alpha))$,  
\item (Sparr) $p,q\geq 1$, $(L_p(\Omega),L_q(\Omega))$ for any $\sigma$-finite measure space $\Omega$.
\item (Sparr, Cwikel) $p,q\in (0,\infty)$, $(L_p(0,\alpha),L_q(0,\alpha))$,
\item (Cwikel) $p\in (0,\infty]$, $(\ell^p, \ell^{\infty})$.
\end{itemize}
\end{theorem}

The definition of left-$p$-monotonicity and right-$q$-monotonicity of a space $E$, contrary to $K$-monotonicity for a couple $(A,B)$, does not impose a condition of the type $A \cap B \subset E \subset A + B$. In the following lemma, we show that this condition is automatically verified.

\begin{lemma}\label{lem:inclusions}
Let $\infty > q\geq p>0$ and $C>0$. Let $E$ be a quasi-Banach symmetric function space. Then:
\begin{enumerate}
\item if $E$ is left-$p$-monotone and right-$q$-monotone, $L_p \cap L_q \subset E \subset L_p + L_q,$
\item if $E$ is left-$p$-monotone, $L_p \cap L_\infty \subset E \subset L_p + L_\infty$.
\end{enumerate}
\end{lemma}

\begin{proof}
We only prove (1) since (2) can be obtained by similar arguments. Let us prove the first inclusion. Since $E$ is a symmetric function space, it contains  characteristic functions of sets of finite measure (see \cite{Dir11}). Let $f\in (L_p \cap L_q)^+$. Decompose $f$ into $f_1 = f\Ind_{\set{f>1}}$ and $f_2 = f\Ind_{\set{f\leq 1}}$. Set,
$$h_1 = \frac{\Norm{f_1}_q}{\nu(\set{f>1})^{1/q}}\Ind_{\set{f>1}}\ \text{and}\ h_2 = \Ind_{(0,\Norm{f_2}_p^p)}.$$
Since $h_1^q$ is the mean of $f_1^q$ (on the support of $f_1$), $h_1^q \triangleright f_1^q$. Hence, $f_1\in E$.
Now let us show that $h_2^p \succ f_2^p$. If $t< \Norm{f_2}_p^p$,
$$\int_0^t h_2^p = \int_0^t 1 \geq \int_0^t (f_2^*)^p,$$
and if $t\geq \Norm{f_2}_p^p$,
$$\int_0^t h_2^p = \Norm{f_2}_p^p = \int_0^\alpha (f_2^*)^p \geq \int_0^t (f_2^*)^p.$$
Consequently, $f_2\in E$. So $f$ belongs to $E$.

Let us now prove the second inclusion. Let $f\notin L_p + L_q$. Decompose $f$ into $f_1$ and $f_ 2$ as above. By assumption, $f_1 \notin L_p + L_q$ or $f_2 \notin L_p + L_q$. 

\emph{case 1: $f_1 \notin L_p + L_q$.} Then $f_1$ is not in $L_p$, this means that $f_1^p \succ a\Ind_{[0,\min(1,\alpha)]}$ for all $a >0$, hence $\Norm{f_1}_E = \infty$, $f_1 \notin E$. 

\emph{case 2: $f_2 \notin L_p + L_q$.} Similarly, $f_2 \notin L_q$ so $f_2^q \triangleright a\Ind_{[0,\min(1,\alpha)]}$ for all $a>0$ so $f_2 \notin E$.
\end{proof}

\subsection{Boyd indices}\label{subsection Boyd}

Boyd indices play an important role in describing the boundedness properties of classical operators on symmetric spaces (\cite{Boy69}, \citep{KalMon03}).
They intervene in our study as a handy tool to connect some properties of quasi-Banach symmetric spaces to interpolation.
Let $E$ be a symmetric function space on $(0,\alpha)$. 

\begin{definition}
Let $t>0$. We denote by $D_t$ the dilation operator on $E$ given by:
$$\begin{array}{ccccc}
D_t & : & E & \to & E \\
 & & f & \mapsto & [s \mapsto f(s/t)] \\
\end{array},$$
where by convention $f(s) = 0$ if $s\geq \alpha$. The Boyd indices of $E$ are given by:
$$
\alpha_E = \lim_{t\to 0} \dfrac{\log \Norm{D_t}_{E\to E}}{\log t} = \sup_{t<1} \dfrac{\log \Norm{D_t}_{E\to E}}{\log t},
$$
and
$$
\beta_E = \lim_{t\to \infty} \dfrac{\log \Norm{D_t}_{E\to E}}{\log t} = \inf_{t>1} \dfrac{\log \Norm{D_t}_{E\to E}}{\log t}.
$$
\end{definition}

\begin{remark}
{\rm Since $\alpha_E$ is defined as a supremum, if $\alpha_E=0$ then $\Norm{D_t}_{E\to E} = 1$ for all $t\in (0,1]$.}
\end{remark}

\begin{remark}\label{rem:beta}
{\rm Since $E$ is a quasi-Banach space, $\beta_E < \infty$.}
\end{remark}

\begin{proof}
By definition, there exists a constant $C>0$ such that for any $f,g\in E$, $\Norm{f+g}_E \leq C(\Norm{f}_E + \Norm{g}_E)$. Consequently, $\Norm{D_{2}}_{E\to E} \leq 2C$ and $\beta_E \leq \frac{\log 2C}{\log 2}$.
\end{proof}

For any $p,q\in (0,\infty)$, consider the map:
$$\begin{array}{ccccc}
H^{(p)} & : & E & \to & E \\
 & & f & \mapsto & \left[ s \mapsto \p{\frac{1}{s}\int_0^s (f^*)^p}^{1/p} \right] \\
\end{array},$$
and:
$$\begin{array}{ccccc}
H_{(q)} & : & E & \to & E \\
 & & f & \mapsto & \left[ s \mapsto \p{\frac{1}{s}\int_s^\infty (f^*)^q}^{1/q} \right] \\
\end{array}.$$
The following lemma is a particular case of \cite[Theorem 2]{Mon96}.
\begin{lemma}\label{lem : Montgomery}
Let $E$ be a symmetric function space such that $p < \frac{1}{\beta_E}$ and $q > \frac{1}{\alpha_E}$. Then $H^{(p)}$ and $H_{(q)}$ are bounded on $E$.
\end{lemma}

\begin{remark}
In \cite{Mon96}, contrary to the definition we took, symmetric spaces are automatically assumed to verify some version of the Fatou property. However, the proof of Lemma \ref{lem : Montgomery} given in \cite{Mon96} does not use this hypothesis so we do not reproduce it here.
\end{remark}

As a consequence of Lemma \ref{lem : Montgomery}, we obtain an interpolation result which can already be found in S. Dirksen's PhD thesis \cite{Dir11}.

\begin{property}\label{prop:boydintinfty}
Assume that $p < \frac{1}{\beta_E}$. Then, $E$ is an interpolation space for the couple $(L_p,L_{\infty})$.
\end{property}

\begin{proof}
By Lemma \ref{lem : Montgomery}, $H^{(p)}$ is well-defined and bounded from $E$ to $E$, in particular $E \subset L_p + L_{\infty}$. Let $f\in E$ and $g\in L_p + L_{\infty}$ and assume that $f^p \succ g^p$. This means that $H^{(p)}(f) \geq H^{(p)}(g) \geq g^*$. Since $H^{(p)}(f)$ belongs to $E$ this means that $g$ belongs to $E$ and that:
$$\Norm{g}_E \leq \Norm{H^{(p)}(f)}_E \les \Norm{f}_E.$$
Hence, $E$ is $K$-monotone for the couple $(L_p,L_{\infty})$. This concludes the proof by Proposition \ref{prop:KmonToInter}.
\end{proof}

Combining the previous proof and Remark \ref{rem:beta}, we can obtain the following.

\begin{corollary}\label{cor:allpmon}
There exists $p>0$ such that $E$ is is left-$p$-monotone.
\end{corollary}

Recall also the following classical result \cite[Theorem 3]{Mon96}. Once again, the proof given in \cite{Mon96} applies without modification to our context since it only uses the boundedness of $H^{(p)}$ and $H_{(q)}$ combined with \eqref{formula:KfuncLpLq}.

\begin{property}\label{prop:boydint}
Assume that $0 < p < \frac{1}{\beta_E} \leq \frac{1}{\alpha_E} < q \leq \infty$. 
Then, $E$ is an interpolation space for the couple $(L_p,L_q)$.
\end{property}

The next lemma will be needed in the last section of the paper. We follow \cite[Lemma 7.2]{Ast09} but we provide a more precise statement.

\begin{lemma}\label{lem:alpha=0}
Suppose that $\alpha_E = 0$. Let $\varphi \in L_{\infty}(0,1)$ consider the map:
$$\begin{array}{ccccc}
T_\varphi & : & E & \to & E((0,1) \times (0,\alpha)) \\
 & & f & \mapsto & [(s,t) \to \varphi(s)f(t)] \\
\end{array}.$$
Then, $\normop{T_\varphi} = \Norm{\varphi}_{\infty}$.
\end{lemma}

\begin{proof}
It is clear that $\normop{T_\varphi} \leq \Norm{\varphi}_{\infty}$. Let $ a < \Norm{\varphi}_{\infty}$. This means that 
$\nu\p{ \set{\md{\varphi} > a}} = b > 0$. 
Remark that for any $g \in E$ and $t>0$, 
$$\nu\p{\set{\md{T_{a\Ind_{\md{\varphi}>a}}(g)} > t}} = \nu\p{\set{\md{\varphi}>a}\times\set{\md{ag}>t}} = b\nu\p{\set{\md{ag}>t}}.$$
And since $b\leq 1$,
$$\nu\p{\md{aD_b(g)}>t} = b\nu\p{\set{\md{ag}>t}}.$$
Hence, $\mu (T_{a\Ind_{\md{\varphi} > a}}(g)) = a\mu (D_b(g))$. 

Since $\alpha_E = 0$, $\Norm{D_b}_{E\to E} = 1$. Let $\e >0$, there exists $x \in E$ such that $\Norm{x}_E = 1$ and $\Norm{D_b(x)}_E \geq 1-\e$. Hence,
$$\normop{T_\varphi} \geq \Norm{T_\varphi(x)}_E \geq \Norm{T_{a\Ind_{\md{\varphi} > a}}(x)}_E = a\Norm{D_b(f)}_E \geq a(1-\e).$$
Since this is true for any $a<\Norm{\varphi}_{\infty}$ and $\e>0$, we can conclude that $\normop{T_\varphi} \geq \Norm{\varphi}_{\infty}$.
\end{proof}

\section{Interpolation spaces for the couple $(L_p,L_{\infty})$}\label{3}

Recall that $\alpha \in (0,\infty]$ is fixed and that for any $p\in [0,\infty]$ we write $L_p = L_p(0,\alpha)$. We are interested in the following extension of the Lorentz-Shimogaki theorem to $p<1$:

\begin{theorem} \label{thm:LorentzShimogakiquasiBanach}
Let $p<1$ and $E$ be a symmetric function space. Then $E$ is an interpolation space for the couple $(L_p,L_{\infty})$ if and only if $E$ is left-$p$-monotone.
\end{theorem}

We will deduce the theorem above from the following proposition:

\begin{property}\label{prop:LS:f mapsto g if f succ g}
Let $f,g \in L_p + L_{\infty}$ and $d>0$. Suppose that $f$ and $g$ are non-negative, non-increasing, left-continuous, take values in $\Db = \{d^n\}_{n\in\Zb} \cup \{0\}$ and $f^p \succ g^p$. Then there exists an operator $T: L_p + L_{\infty} \to L_p + L_{\infty}$ and a positive function $h$ such that:
\begin{itemize}
\item $h \leq 2^{1/p}f$
\item $Th^* = g$, 
\item the restrictions of $T$ on $L_p$ and $L_{\infty}$ are contractions.
\end{itemize}
\end{property}

\begin{proof}
Define $h := 2^{1/p}f\Ind_{2f^p\geq g^p}$. Let $t>0$, let us show that $\Int{0}{t} h^p \geq \Int{0}{t} g^p$. Note that $g^p > 2f^p \Leftrightarrow g^p < 2(g^p - f^p)$.
\begin{align*}
\Int{0}{t} h^p - \Int{0}{t} g^p &= \Int{0}{t} \Ind_{2f^p\geq g^p} (2f^p - g^p) - \Ind_{2f^p < g^p} g^p \\
&\geq \Int{0}{t} \Ind_{2f^p\geq g^p} 2(f^p - g^p) - \Ind_{2f^p < g^p} 2(g^p - f^p) \\
&= 2\Int{0}{t} f^p - g^p \\
&\geq 0.
\end{align*}
Denote $X := \Norm{g}_p \in (0,\infty]$. Define: 
\begin{center}
$H:t \mapsto \Int{0}{t} h^p$ and $G:t \mapsto \Int{0}{t} g^p$. 
\end{center}
Let $x\in [0,X)$, note that $G^{-1}(x)$ is well-defined since $G$ is continuous and increasing on $[0,\beta]$, the support of $g$. Define also $H^{-1}(x) = \min H^{-1}(\{x\})$ ($H$ is continuous). Note that $h(H^{-1}(x)) \neq 0$. Indeed, by definition of $t_x := H^{-1}(x)$, for all $\e>0$, $\Int{t_x-\e}{t_x} h \neq 0$ so there exists an increasing sequence $(t_n)$ converging to $t_x$ and such that $h(t_n) \neq 0$. This means that $2f(t_n)\geq  g(t_n)$. Using the fact that $f$ is left-continuous and $g$ is non-increasing, we obtain $2f(t_x) \geq g(t_x)$. And since we showed that $H \geq G$, $H^{-1}(x) \leq G^{-1}(x)$, hence $g^p(H^{-1}(x)) \geq g^p(G^{-1}(x))$. In conclusion:
\[
h^p(H^{-1}(x)) \geq g^p(G^{-1}(x)). \addtag \label{eq:LS:h>g}
\]
Consider now $h^*$ and $H_*: t \mapsto \Int{0}{t} (h^*)^p$. By construction, $h$ restricted to $\{h>0\}$ is non-increasing and $h^*$ is non-increasing too. 
Since $h$ and $h^*$ also have a countable range, they are of the form $\Sum{i\in\Nb}{} a_i \Ind_{I_i}$ where the $I_i$ are intervals. 
Let $t$ in the support of $h$ so that $t$ is not an end point of one of those intervals then since $h$ is non-increasing: $$\Int{h>h(t)}{} h < \Int{0}{t} h < \Int{h\geq h(t)}{} h.$$
Similarly, if $t'$ is in the support of $h^*$ and not at an end point,
$$\Int{h^*>h^*(t')}{} h^* < \Int{0}{t'} h^* < \Int{h^*\geq h^*(t')}{} h^*.$$
Since $h$ and $h^*$ have the same distribution, this means that for all $s>0$, 
\begin{center}
$\Int{h>s}{} h = \Int{h^*>s}{} h^*$ and $\Int{h\geq s}{} h = \Int{h^*\geq s}{} h^*$.
\end{center}
By the previous inequalities, we deduce that except for a countable number of $t$ and $t'$, if $H(t) = H_*(t')$ then $h(t) = h^*(t')$. Since $H_*$ is strictly increasing on $H_*^{-1}((0,X))$, this implies by \eqref{eq:LS:h>g} that for all $x\in (0,X)$, except maybe a countable number, 
\[
h^*(H_*^{-1}(x)) = h(H^{-1}(x)) \geq g(G^{-1}(x)).\addtag\label{eq:LS:h*>g}
\]
Let $\phi = H_*^{-1} \circ G$ on $(0,\beta)$. Remark that since $h^*$ and $g$ have range in $2^{1/p}\Db$ and $\Db$ respectively, $H_*$ and $G$ are piecewise affine on every compact of $(0,\beta)$. Since $H_* \circ \phi = G$, by derivating, for all $t\in (0,\alpha)$, except maybe a countable number, 
\[ 
\phi'(t)(h^*\circ\phi)^p(t) = g^p(t). \addtag\label{eq:LS:formula for T}
\]
Hence define $T: L_p + L_{\infty} \to L_p + L_{\infty}$ by $T(u) = (\phi'^{1/p}\cdot u\circ\phi)\circ\Ind_{(0,\beta)}$. By the change of variables formula, $T$ is a contraction on $L_p$. Indeed, one can apply the formula on any compact of $(0,\beta)$ and obtain the $p$-norm as a limit. Let $t\in (0,\alpha)$ and write $t = G^{-1}(x)$, then by \eqref{eq:LS:h*>g}:
$$\phi'(t) = (H_*^{-1} \circ G)'(t) = \dfrac{g^p(t)}{(h^*)^p(H_*^{-1}(G(t))} = \p{\dfrac{g(G^{-1}(x))}{h^*(H_*^{-1}(x))}}^p \leq 1,$$
except for a countable number of values of $t$. So $T$ is also a contraction on $L_{\infty}$. Finally, by \eqref{eq:LS:formula for T}, $Th^* = g$.
\end{proof}

\begin{proof}[Proof of Theorem \ref{thm:LorentzShimogakiquasiBanach}]
The reverse implication is a consequence of Proposition \ref{prop:KmonToInter} and Lemma \ref{lem:inclusions}. So let us focus on the direct implication. Let $E$ be a interpolation space for the couple $(L_p, L_{\infty})$ with constant $C$. Let $f\in E$ and $g\in L_p + L_{\infty}$ such that for all $t>0$, $K_t(f,L_p,L_{\infty})\geq K_t(g,L_p,L_{\infty})$. By \eqref{formula:KtLpLinfty}, we can replace this condition by $f^p \succ g^p$. Let $d \in (1,\infty)$ and define $f_2 = d^{\Floor{\log_d(f^*)}+1}$ and $g_2 = d^{\Floor{\log_d(g^*)}}$. We have $ f^* \leq f_2 \leq df^*$, $g^*\leq dg_2 \leq dg^*$, hence $f_2^p \succ g_2^p$. Furthermore, $f_2$ and $g_2$ take values in $\Db = \{d^n\}_{n\in\Zb} \cup \{0\}$. Take $h$ and $T$ given by Proposition \ref{prop:LS:f mapsto g if f succ g}. Since $h\leq 2^{1/p} f_2$, $h\in E$ so by the definition of an interpolation space, $T(h^*) = g_2 \in E$ and finally $g \in E$. Furthermore:
\begin{align*}
\Norm{g}_E &= \Norm{g^*}_E \leq d\Norm{g_2}_E \leq dC\Norm{h}_E \\
&\leq dC2^{1/p} \Norm{f_2}_E \leq d^2 C2^{1/p}\Norm{f^*}_E = d^2 C2^{1/p}\Norm{f}_E.
\end{align*}
This is true for any $d\in (1,\infty)$ so, $\Norm{g}_E \leq C2^{1/p}\Norm{f}_E$.
\end{proof}

\section{Interpolation spaces for the couple $(L_p,L_q)$}\label{4}

\subsection{Function spaces}

In this section, we prove the main theorem of this paper, a characterization of interpolation spaces between $L_p$ and $L_q$. The proof uses a strategy similar as \cite{CwiNil18}. Let us first state the result precisely.

\begin{theorem}\label{thm:functionspace}
Let $0<p\leq q<\infty$ and $E$  be a quasi-Banach symmetric function space. Then $E$ is left-$p$-monotone and right-$q$-monotone if and only if $E$ is an interpolation space between $L_p$ and $L_q$.
\end{theorem}

\begin{proof}[Proof of the reverse implication]
Suppose that $E$ is an interpolation for the pair $(L_p,L_q)$. Then by reiteration, $E$ is an interpolation space for the pair $(L_p,L_{\infty})$ and by Theorem \ref{thm:LorentzShimogakiquasiBanach}, $E$ is left-$p$-monotone. Let $f\in E$ and $g\in L_p + L_q$ such that $f^q \triangleright g^q$. Assume that $f$ and $g$ are positive by taking if necessary their modules. We will show that $g\in E$ and $\Norm{g}_E \les \Norm{f}_E$. Recall that $(L_p,L_q)$ is a Calder\'on couple (Theorem \ref{thm:Kmon for Lp literature}), so it is enough to show that for all $t>0$,
$$K_t(g,L_p,L_q) \les K_t(f,L_p,L_q).$$ 
Recall also that, for all $h\in L_p + L_q$ and $t>0$ (\cite{Hol70}):
$$K_t(h,L_p,L_q) \approx \p{\Int{0}{t^{r}} (h^*)^p}^{1/p} + t \p{\Int{t^{r}}{\infty} (h^*)^q}^{1/q},$$
where $r = (p^{-1} - q^{-1})^{-1}$.
Let $t>0$. First, since $f^q \triangleright g^q$, 
$$t \p{\Int{t^{r}}{\infty} (g^*)^q}^{1/q} \leq t \p{\Int{t^{\alpha}}{\infty} (f^*)^q}^{1/q} \leq K_t(f,L_p,L_q).$$
We now have to estimate the other term i.e to prove that:
$$ \p{\Int{0}{t^{r}} (g^*)^p}^{1/p} \les K_t(f,L_p,L_q).$$ 
Suppose now that $f$ and $g$ are bounded so that $f^q$ and $g^q$ belong to $L_1$. Using again the fact that $f^q \triangleright g^q$, it is easy to see that there exists $g'\geq g$ such that $\Norm{g'}_q = \Norm{f}_q$ and $g'^q \succ f^q$ (for example by letting $g' = g + h$ where $h$ is a function taking only the values $\Norm{f}_{\infty} +1$ and $0$ and $\Norm{g'}_q = \Norm{f}_q$). Then by \cite{Hia87} (Theorem 4.7, (1)), there exists a unital, positive, integral preserving operator $T$ such that $T((g'^*)^q) = (f^*)^q$. Hence $T((g^*)^q) \leq (f^*)^q$. Define $e = \Ind_{(0,t^{r})}$ and write,
$$\p{\Int{0}{t^{r}} (g^*)^p}^{1/p} = \p{\Int{0}{\infty} e(g^*)^p}^{1/p} = \p{\Int{0}{\infty} T(e(g^*)^p)}^{1/p}.$$
Now, apply Lemma \ref{lem:holder} with $s = q/p$ and $s'=r/p$ to obtain:
\begin{align*}
\p{\Int{0}{t^{r}} (g^*)^p}^{1/p} &\leq \p{\Int{0}{\infty} T(e)^{p/r}T((g^*)^q)^{p/q}}^{1/p} \\
&\les \p{\Int{0}{t^{r}} (f^*)^pT(e)^{p/r}}^{1/p} + \p{\Int{t^{r}}{\infty} (f^*)^pT(e)^{p/r}}^{1/p} \\
\intertext{To estimate the left summand, we use the fact that $\Norm{Te}_{\infty} \leq 1$ and for the right summand, we apply H\"older's inequality.}
&\leq \p{\Int{0}{t^{r}} (f^*)^p}^{1/p} + \p{\Int{t^{r}}{\infty} (f^*)^q}^{1/q}\p{\Int{t^{r}}{\infty} T(\Ind_{(0,t^{r})})}^{1/r} \\
&\les K_t(f,L_p,L_q)
\end{align*}
where we used that $\p{\Int{t^{r}}{\infty} T(\Ind_{(0,t^{r})})}^{1/r} \leq \p{\Int{0}{\infty} T(\Ind_{(0,t^{r})})}^{1/r} = t$. To conclude for unbounded functions, it suffices to approximate $g^*$ and $f^*$. A way to do it is to apply the previous inequality to $g^*(.-\e)$ and $f^*(.-\e)$ and let $\e$ go to $0$.
\end{proof}

For the direct implication, the following lemma is the key argument. It enables us to relate left-$p$-monotonicity and right-$q$-monotonicity to $K$-monotonicity for the couple $(L_p,L_q)$.

\begin{lemma} \label{lem:technique:fKg}
Let $0<p\leq q<\infty$, $r>0$ and $f,g\in (L_p + L_q)^+$ be non-increasing right-continuous functions such that for all $t>0$:
$$\p{\Int{0}{t^{r}} f^p}^{1/p} + t \p{\Int{t^{r}}{\infty} f^q}^{1/q} \geq \p{\Int{0}{t^{r}} g^p}^{1/p} + t \p{\Int{t^{r}}{\infty} g^q}^{1/q},$$
then there exist nonegative functions $h$ and $l$ such that $h + l = g$, $h$ is non-increasing, $f^p \succ h^p$ and $f^q \triangleright l^q$. More precisely, for all $t>0$,
$$\Int{t}{\infty} f^q \geq \Int{t}{\infty} l^q.$$
\end{lemma}

\begin{proof}
Let $A := \{ t\in \Rb^+ : \Int{0}{t} f^p \geq \Int{0}{t} g^p \}$ and $B := \{ t\in \Rb^+ : \Int{t}{\infty} f^q \geq \Int{t}{\infty} g^q \}$. $A$ and $B$ are closed sets and $A \cup B = (0,\infty)$, denote $A^c = \Rb^+ \backslash A$. For $t\in \Rb^+$, set $a(t) = \min A \cap [t,\infty)$ ($a(t) = \infty$ if $A \cap [t,\infty) = \emptyset$) and for convenience, $g(\infty) = 0$. 

Define $h: t \mapsto g(a(t))$ and $l = g - h$, $g$ and $l$ are clearly non-negative.

Note that for $t\in A^c$, if $a(t) \neq \infty$ then $f(a(t)) \geq g(a(t))$. Indeed, assume by contradiction that $f(a(t)) < g(a(t))$. By left-continuity of $f$ and $g$ there exists $s \in (t, a(t))$ (in particular, $s\notin A$) such that $f < g$ on $(s,a(t))$. Hence:
$$\int_0^{a(t)} f^p = \int_0^s f^p + \int_s^{a(t)} f^p < \int_0^s g^p + \int_s^{a(t)} g^p = \int_0^{a(t)} g^p,$$
which contradicts the fact that $a(t) \in A$. This implies that $f(t) \geq h(t)$ for any $t\in A^c$. 

Let us now check that the decomposition verifies what we claimed. Let $t\in\Rb^+$ and let $b(t) = \max A \cap [0,t]$ (remark that $0\in A$ so $b$ is well-defined). Clearly $(b(t),t)\subset A^c$ and $ b(t) \in A$. Hence, 
$$\Int{0}{t} h^p = \Int{0}{b(t)} h^p + \Int{b(t)}{t} h^p 
\leq \Int{0}{b(t)} g^p + \Int{b(t)}{t} f^p 
\leq \Int{0}{t} f^p.$$
Finally, set $c(t) = \inf A^c \cap [t,\infty)$ ($c(t) = \infty$ if $A^c \cap [t,\infty) = \emptyset$). Note that since $B \cup A = \Rb^+$, we always have $c(t) \in B$ and remark that $l$ is supported in $A^c$. Hence,
$$\int_t^{\infty} l^p \leq \int_{c(t)}^\infty g^p \leq \int_{c(t)}^{\infty} f^p \leq \int_t^\infty f^p,$$
and since $l$ is positive:
$$\Int{t}{\infty} (l^*)^q \leq \Int{t}{\infty} l^q \leq \Int{t}{\infty} f^q = \Int{t}{\infty} (f^*)^q.$$
\end{proof}

We are now ready to conclude the proof of the main theorem.

\begin{proof}[End of the proof of Theorem \ref{thm:functionspace}]
By Proposition \ref{prop:KmonToInter}, it suffices to prove that $E$ is $K$-monotone. We already know from Lemma \ref{lem:inclusions} that $L_p \cap L_q \subset E \subset L_p + L_q$. Assume that $f \neq 0$. Suppose that for all $t>0$, $K_t(f,L_p,L_q) \geq K_t(g,L_p,L_q)$. By the formula (\ref{formula:KfuncLpLq}) for the $K$-functional of the couple $(L_p,L_q)$, there exists a constant $c>0$ independent of $f$ and $g$ such that for all $t>0$:
$$\p{\Int{0}{t^{\alpha}} (cf^*)^p}^{1/p} + t \p{\Int{t^{\alpha}}{\infty} (cf^*)^q}^{1/q} \geq \p{\Int{0}{t^{\alpha}} (g^*)^p}^{1/p} + t \p{\Int{t^{\alpha}}{\infty} (g^*)^q}^{1/q},$$
So by Lemma \ref{lem:technique:fKg} applied to $cf^*$ and $g^*$, there exists two functions $h$ and $l$ such that $h + l = g^*$, $(cf)^p \succ h^p$ and $(cf)^q \triangleright l^q$. Hence, by the assumptions on $E$, $h,l\in E$ so $g\in E$ and 
$$\Norm{g}_E = \Norm{g^*}_E \les \Norm{h}_E + \Norm{l}_E \leq 2cK\Norm{f}_E.$$
\end{proof} 

We can now obtain the theorem stated in the introduction as a corollary.

\begin{proof}[Proof of Theorem \ref{thm:intro:conjfonc}]
$(1) \Rightarrow (2)$. By corollary \ref{cor:allpmon} there exists $p$ which can be chosen to be less than $q$ such that $E$ is left-$p$-monotone. Hence by Theorem \ref{thm:functionspace}, $E \in Int(L_p,L_q)$.

$(2) \Rightarrow (1)$. This is true by Theorem \ref{thm:functionspace}.
\end{proof}

\begin{corollary} \label{cor:alphaqmon}
Let $E$ be a symmetric quasi-Banach function space. Assume that $\alpha_E \neq 0$. Then, for any $q \in (1/\alpha_E,\infty)$, $E$ is right-$q$-monotone.
\end{corollary}

\begin{proof}
By Remark \ref{rem:beta} and Proposition \ref{prop:boydint}, $E \in Int(L_p,L_q)$ for some $p\leq q$.
\end{proof}

\begin{remark}
Let us define, as a convention, that every symmetric space $E$ is right-$\infty$-monotone. Then, Theorem \ref{thm:functionspace} holds for $q = \infty$ and generalises Theorem \ref{thm:LorentzShimogakiquasiBanach}.
\end{remark}

\subsection{An application to convexifications}

Recall that if $E$ is a symmetric function space, the $r$-{\it convexification} of $E$ is given by:
$$E^{(r)} = \set{f \in L_0(0,\alpha) : \md{f}^{r} \in E},\ \text{with} \ \Norm{f}_{E^{(r)}} = \Norm{\md{f}^r}_E.$$
\begin{lemma}\label{lem: monotonicity of convex}
Let $E$ be a symmetric function space. Let $p,q,r\in (0,\infty)$. Then:
\begin{itemize}
\item $E$ is left-$p$-monotone if and only if $E^{(r)}$ is left-$rp$-monotone,
\item $E$ is right-$q$-monotone if and only if $E^{(r)}$ is right-$rq$-monotone.
\end{itemize}
\end{lemma}

\begin{proof}
Let $f\in E^{(r)}$ and $g\in L_0(0,\alpha)$ such that $\md{f}^{rp} \succ \md{g}^{rp}$. By definition $\md{f}^r \in E$ so by left-$p$-monotonicity of $E$, $\md{g}^r \in E$ and $\Norm{\md{g}^r}_E \les \Norm{\md{f}^r}_E$. This exactly means that $g\in E^{(r)}$ and $\Norm{g}_{E^{(r)}} \les \Norm{f}_{E^{(r)}}$.

The rest of the statement is checked similarly.
\end{proof}

We can now deduce the following corollary, which extends a result of Montgomery-Smith to quasi-Banach spaces. 

\begin{corollary}
Let $p<q \in (0,\infty]$ and $r\in (0,\infty)$. Let $E$ be a symmetric function space. Then the following are equivalent:
\begin{enumerate}
\item $E$ is an interpolation space for the couple $(L_p,L_q)$,
\item $E^{(r)}$ is an interpolation space for the couple $(L_{rp},L_{rq})$.
\end{enumerate}
\end{corollary}

\begin{proof}
If $q\neq 0$, this is immediate by Lemma \ref{lem: monotonicity of convex} and Theorem \ref{thm:functionspace}. If $q=0$, this is also direct using Lemma \ref{lem: monotonicity of convex} and Theorem \ref{thm:LorentzShimogakiquasiBanach}.
\end{proof}

\subsection{Sequence spaces}

We will now prove the main theorem in the context of symmetric sequence spaces rather than symmetric function spaces. We will use the fact that $\ell^{\infty}$ can be embedded into $L_{\infty}$ by:
$$u \mapsto \Sum{i=1}{\infty} u_i\Ind_{[i-1,i)}.$$
Hence, we will identify $\ell^{\infty}$ with the subalgebra of $L_{\infty}$ of functions a.e. constant on intervals of the form $[i,i+1)$, $i\in\Nb$. Denote by $\Eb$ the conditional expectation from $L_{\infty}$ to $\ell^{\infty}$. The notions of left and right majorization can be extended without modifications to sequences and coincide with the usual notion. Let $u$ and $v$ in $\ell^{\infty}$:
\begin{center}
$u\succ v \Leftrightarrow \forall n\in\Nb, \Sum{i=1}{n} u_i^* \geq \Sum{i=1}{n} v_i^*$ and 
$u\triangleright v \Leftrightarrow \forall n\in\Nb, \Sum{i=n}{\infty} u_i^* \geq \Sum{i=n}{\infty} v_i^*.$
\end{center}
Let us now state the result of this subsection: 

\begin{theorem}\label{thm:sequencespace}
Let $0<p\leq q<\infty$ and $E$ be a quasi-Banach symmetric sequence space. If $E$ is left-$p$-monotone and right-$q$-monotone then $E$ is an interpolation space between $\ell_p$ and $\ell_q$.
\end{theorem}

\begin{proof}
By Proposition \ref{prop:KmonToInter}, it suffices to check that $E$ is $K$-monotone. Similarly to the case of function spaces, the hypothesis of the theorem imply that $\ell^p = \ell^p \cap \ell^q \subset E \subset \ell^p + \ell^q = \ell^q$ and that for any $u\in E$ and $v\in \ell^q$:
\begin{center}
$u^q \triangleright v^q \Rightarrow v\in E, \Norm{v}_E \leq C\Norm{u}_E$ and $u^p \succ v^p \Rightarrow v\in E, \Norm{v}_E \leq C\Norm{u}_E.$
\end{center}
Let $u\in E$ and $v\in \ell^q$, $u\neq 0$. Assume that for all $t>0,K_t(u,\ell^p,\ell^q) \leq K_t(v,\ell^p,\ell^q)$. The formula for the $K$-functional used previously still holds in this context. Indeed, for any sequence $w\in\ell^q$ and $t>0$:
$$K_t(w,p,q) \approx \p{\Int{0}{t^{\alpha}} (w^*)^p}^{1/p} + t \p{\Int{t^{\alpha}}{\infty} (w^*)^q}^{1/q}.$$
So there exists a constant $c$ independent of $u$ and $v$ so that for all $t>0$:
$$\p{\Int{0}{t^{\alpha}} (cu^*)^p}^{1/p} + t \p{\Int{t^{\alpha}}{\infty} (cu^*)^q}^{1/q} \geq \p{\Int{0}{t^{\alpha}} (v^*)^p}^{1/p} + t \p{\Int{t^{\alpha}}{\infty} (v^*)^q}^{1/q},$$
By Lemma \ref{lem:technique:fKg}, there exists nonegative functions $h$ and $l$ such that, $h + l = v^*$, $h$ is non-increasing, $(cu)^p \succ h^p$ and $(cu)^q \triangleright l^q$. Let $\overline{h} = \Eb(h^p)^{1/p}$ and $\overline{l} = \Eb(l^q)^{1/q}$. 

We will start by showing that $(cu)^p \succ \overline{h}^p$ and $(cu)^q \triangleright \overline{l}^q$.
The functions:
\begin{center}
$t \mapsto \Int{0}{t} (cu^*)^p$ and $t \mapsto \Int{0}{t} \overline{h}^p$
\end{center}
are affine on intervals of the form $[i,i+1]$, $i\in\Nb$. So it suffices to check the inequality at integers. Let $n\in \Nb$,
$$\Int{0}{n} (cu^*)^p \geq \Int{0}{n} h^p = \Int{0}{n} \overline{h}^p.$$
The other inequality is proven in the same manner. 

Recall that $\overline{h}$ and $\overline{l}$ can be seen as sequences. Note that $a_p := \max(1,2^{p-1})$ is the optimal constant such that for all $a,b\in\Rb^+$, $(a+b)^p \leq a_p(a^p + b^p)$. We have:
$$v^* = \Eb((v^*)^p)^{1/p} \leq a_p\Eb(h^p + l^p)^{1/p} \leq a_pa_{1/p}\p{\Eb(h^p)^{1/p} + \Eb(l^p)^{1/p}} 
\leq  a_pa_{1/p}(\overline{h} + \overline{l}),$$
where we used the fact that $\Eb(l^p)^{1/p} \leq \Eb(l^q)^{1/q}$ since $q\geq p$. By hypothesis on $E$, $\overline{h}\in E$ and $\overline{l}\in E$ with 
\begin{center}
$\Norm{\overline{h}}_E \les \Norm{u}_E$ and $\Norm{\overline{l}}_E \les \Norm{u}_E$.
\end{center}
So $v\in E$ and $\Norm{v}_E \les \Norm{u}_E$ with an implied constant depending on $E$, $p$ and $q$.
\end{proof}

\begin{remark}
As was pointed out to me by M. Cwikel, the couple $(\ell^p,\ell^q)$ is not known in general to be a Calder\'on couple. This is why we cannot show the reverse implication in Theorem \ref{thm:sequencespace}. However, for $1\leq p\leq q \leq \infty$, the couple $(\ell^p,\ell^q)$ is a Calder\'on couple (\cite{Spa78}) and we can recover the following:
\begin{corollary}
Let $1\leq p\leq q<\infty$ and $E$ be a quasi-Banach symmetric sequence space. Then $E$ is left-$p$-monotone and right-$q$-monotone if and only if $E$ is an interpolation space between $\ell_p$ and $\ell_q$.
\end{corollary}
This is what was obtained in \cite{CwiNil18} by Cwikel and Nilsson. 
\end{remark}

\section{A link between $p$-monotonicity and $p$-convexity}\label{5}

In this subsection we prove some lemmas linking $p$-convexity (resp. $q$-concavity) and left-$p$-monotonicity (resp. right-$q$-monotonicity). First, we show that $p$-convexity implies some weak form of the left-$p$-monotonicity. Then, we show that under the assumption that $E$ has the Fatou property, this weak form of left-$p$-monotonicity is, in fact, equivalent to left-$p$-monotonicity by a simple approximation argument. 

Recall that $E$ is said to be $p$-convex with constant $C$ if for all $n\in\Nb$ and $(x_i)_{i\leq n} \in E^n$,
$$\Norm{\p{\Sum{i=1}{n} \md{x_i}^p}^{1/p}}_E \leq C \p{\Sum{i=1}{n} \Norm{x_i}_E^p}^{1/p}.$$
Similarly, $E$ is said to be $q$-concave with constant $C$ if for all $n\in\Nb$ and $(x_i)_{i\leq n} \in E^n$,
$$\p{\Sum{i=1}{n} \Norm{x_i}_E^q}^{1/q} \leq C \Norm{\p{\Sum{i=1}{n} \md{x_i}^q}^{1/q}}_E.$$

We will denote by $F$ the space of dyadic step functions:
$$F := \{ f\in L_{\infty}(0,\infty) : \exists n\in\Nb, \exists N\in\Nb, \exists a\in(\Rb^+)^N, f = \Sum{i=1}{N}a_i\Ind_{[(i-1)2^{-n},i2^{-n})}\}.$$
Note that this space will also be useful in section \ref{section:khintchine}.

\begin{lemma}\label{lem:technique:fmaj g with norm f = norm g}
Let $f,g\in F$ be non-increasing and $p\in (0,\infty)$. 
\begin{enumerate}
\item If $f^p \succ g^p$ then there exists $h\in F$ such that $h\geq g$, $f^p \succ h^p$ and $\Norm{f}_p = \Norm{h}_p$.
\item If $f^p \triangleright g^p$ then there exists $h\in F$ such that $h\geq g$, $f^p \triangleright h^p$ and $\Norm{f}_p = \Norm{h}_p$.
\end{enumerate}
\end{lemma}

\begin{proof}
We believe that it is not difficult to convince oneself that this lemma is true. We provide a possible construction for $h$ in both cases and leave the details to the reader. Note the by considering $f^p$ and $g^p$, we can suppose that $p=1$. 

$(1)$. Fix $a\in (0,\alpha]$, $a < \infty$ and dyadic, such that $f$ and $g$ are supported in $(0,a)$. For any $s \geq 0$, define,
$$h_s \in F:  t \mapsto \max (g(t),s) \Ind_{[0,\alpha)}.$$
Clearly, $h_s$ is non-increasing and $h_s \geq g$. Choose the unique $s_0$ such that $\Norm{h_{s_0}}_1 = \Norm{f}_1$. It can be checked that $f \succ h_{s_0}$. Hence, $h_{s_0}$ satisfies the conditions of the lemma. 

$(2)$. Since $f,g\in F$, we can choose a dyadic $a\leq \alpha$ such that $f$ and $g$ are constant on $(0,a)$. For any $s \geq 0$, define,
$$h_s = g + s\Ind_{[0,a)} \in F,$$
and proceed as in $(1)$.
\end{proof}

\begin{remark}
Let $p\in (0,\infty)$. If $\Norm{f}_p = \Norm{g}_p < \infty$ then $\md{f}^p \succ \md{g}^p \Leftrightarrow \md{g}^p \triangleright \md{f}^p$.
\end{remark}

\subsection{Convexity implies left-monotonicity}

Our first lemma is a direct consequence of the geometric form of the Schur-Horn theorem (Theorem \ref{thm:schurhorn}).

\begin{lemma}
Let $E$ be a quasi-Banach symmetric function space, $p$-convex with constant $C$. Let $f,g \in F$. Then, 
$$f^p \succ g^p \Rightarrow \Norm{g}_E \leq C\Norm{f}_E.$$ 
\end{lemma}

\begin{proof}
Let $f,g\in F$ suppose that $f^p\succ g^p$. Since $\Norm{.}_E$ is increasing, we can also assume that $\Norm{f}_p = \Norm{g}_p$ by Lemma \ref{lem:technique:fmaj g with norm f = norm g}. Since $f$ and $g$ belong to $F$, there exist $N,n\in \Nb$ and vectors $a = (a_i)_{i\leq N}$ and $b = (b_i)_{i\leq N}$ in $(\Rb^+)^N$ such that:
$$f = \Sum{i=1}{N} a_i\Ind_{[(i-1)2^{-n},i2^{-n})}\ \text{and}\ g = \Sum{i=1}{N} b_i\Ind_{[(i-1)2^{-n},i2^{-n})}.$$
The hypothesis on $f$ and $g$ means that $a^p \succ b^p$ and $\Norm{a}_p = \Norm{b}_p$. This means by the Schur-Horn theorem that $b^p$ is in the convex hull of the permutations $a_{\sigma}^p$ of $a^p$ where for $\sigma \in \mathfrak{S}_N$, we write $a^p_{\sigma} := (a_{\sigma(i)}^p)_{1\leq i\leq N}$. Let 
$$f_{\sigma} = \Sum{i=1}{N} a_{\sigma(i)}\Ind_{[(i-1)2^{-n},i2^{-n})}.$$
This means that there exist non-negative coefficients $(\lambda_{\sigma})_{\sigma\in\mathfrak{S}_N}$ adding up to $1$ such that
$$g^p = \Sum{\sigma\in\mathfrak{S}_N}{} \lambda_{\sigma}f_{\sigma}^p = \Sum{\sigma\in\mathfrak{S}_N}{} \md{\lambda^{1/p}f_{\sigma}}^p.$$
Hence, by $p$-convexity of $E$ and noting that since $E$ is rearrangment invariant, $\Norm{f_{\sigma}}_E = \Norm{f}_E$:
$$\Norm{g}_E \leq C\p{\Sum{\sigma\in\mathfrak{S}_N}{} \Norm{\lambda_{\sigma}^{1/p}f_{\sigma}}_E^p}^{1/p} = C\Norm{f}_E.$$
\end{proof}

Let us now state the key approximation lemma.

\begin{lemma}\label{lem:approxmaj1}
Let $p>0$ and $\e>0$. Let $f,g\in T$ be positive non-increasing functions such that $f^p\succ g^p$. There exist sequences $(f_n)$ and $(g_n)$ of functions in $F$ such that $f_n \uparrow f$ a.e. and $g_n \uparrow g$ a.e. and $(1+\e)f_n^p \succ g_n^p$ for all $n\in\Nb$.
\end{lemma}

\begin{proof}
For any $t\geq 0$ and $n\in \Nb$, consider the interval of the form $[i2^{-n},(i+1)2^{-n})$ containing $t$ and denote by $a_n(t)$ its right extremity. Define 
$$g_n(t) = \left\{
    \begin{array}{ll}
        g(a_n(t))  & \mbox{if } t < 2^n \\
        0     & \mbox{if } t \geq 2^n . \\
    \end{array}
\right.$$
By construction, $g_n\in F$ and $g_n$ is non-increasing. Moreover, $g_n(t) \uparrow g(t)$ for all points of continuity of $g$. Since $g$ is non-increasing, it has only countably many points of discontinuity and $g_n \uparrow g$ a.e. Consider also the sequence $(f_n)_{n\in\Nb}$ similarly defined. Now fix $n$ and let us find $m_n$ such that $(1+\e)f_{m_n} \succ g_n$. To do so, note that since for all $m$:
$$F_m:t\mapsto \Int{0}{t}f_m^p$$
is concave and increasing and 
$$G_n:t\mapsto \Int{0}{t}g_n^p$$
is affine on intervals of the form $[i2^{-n},(i+1)2^{-n})$ and constant for $t$ large enough, it suffices to check that $(1+\e)F_m(i2^n)\geq G_n(i2^n)$ for finitely many $i\in\Nb$. 
Since 
$$F_m(t) \to \Int{0}{t} f^p$$ and $$G_n(t) \leq \Int{0}{t} g^p \leq \Int{0}{t} f^p,$$
for $m_n$ large enough, we have $(1+\e)f_{m_n} \succ g_n$.
Let $i_n = \max (m_n,n)$, the sequences $(f_{i_n})$ and $(g_n)$ verify the requirements of the lemma.
\end{proof}

\begin{lemma}\label{lem:FtoE1}
Let $p>0$ and $C>0$. Let $E$ be a quasi-Banach symmetric function space with the Fatou property such that:
$$\forall f,g\in F, f^p \succ g^p \Rightarrow \Norm{g}_E \leq C\Norm{f}_E.$$
Then $E$ is left-$p$-monotone.
\end{lemma}

\begin{proof}
Since $E$ is rearrangement invariant, we can suppose that $f$ and $g$ are positive non-increasing functions by considering $f^*$ and $g^*$. Let $\e>0$. Take $(f_n)$ and $(g_n)$ two sequences given by lemma $\ref{lem:approxmaj1}$. For all $n\in\Nb$:
$\Norm{g_n}_E \leq (1+\e)C\Norm{f_n}_E \leq (1+\e)C\Norm{f}_E$.
So by the Fatou property, $g\in E$ and $\Norm{g}_E \leq (1+\e)C\Norm{f}_E$.
Since this is true for all $\e >0$, $\Norm{g}_E \leq C\Norm{f}_E.$
\end{proof}

By combining the results above, we obtain the following:

\begin{corollary}\label{cor:convTomonotone}
Let $E$ be a quasi-Banach symmetric sequence space with the Fatou property. If $E$ is $p$-convex then $E$ is left-$p$-monotone.
\end{corollary}

\subsection{Concavity implies right-monotonicity}

The same kind of proofs can be reproduced to obtain results on concavity and right-monotonicity. The approximation lemma, similar to Lemma \ref{lem:approxmaj1}, is given by:

\begin{lemma}\label{lem:approxmaj2}
Let $q>0$ and $\e>0$. Let $f,g\in T$ be positive non-increasing functions such that $f^q\triangleright g^q$. There exist sequences $(f_n)$ and $(g_n)$ of functions in $F$ such that $f_n \uparrow f$ a.e., $g_n \uparrow g$ a.e. and $(1+\e)f_n^q \triangleright g_n^q$ for all $n\in\Nb$.
\end{lemma}

And we have the following:

\begin{lemma}\label{lem:concave}
Let $E$ be a quasi-Banach symmetric function space. 
\begin{itemize}
\item if $E$ is $q$-concave with constant $C$ then for all $f,g \in F$, 
$$f^q \triangleright g^q \Rightarrow \Norm{g}_E \leq C\Norm{f}_E.$$
\item if $E$ has the Fatou property and there exists $C>0$ such that for all $f,g \in F$, 
$$f^q \triangleright g^q \Rightarrow \Norm{g}_E \leq C\Norm{f}_E,$$
then $E$ is right-$q$-monotone.
\item if $E$ is $q$-concave and has the Fatou property, $E$ is right-$q$-monotone.
\end{itemize}
\end{lemma}

\begin{remark}
The reverse of the lemma above and corollary \ref{cor:convTomonotone} are not true. Indeed, consider for example the space $L_{1,\infty}$. It will be proved later that it is right-$q$-monotone for all $q>1$ but it is not $q$-concave for any $q$. Constructing symmetric spaces which are not $p$-convex for any $p$ is not an easy task. We are indebted to F. Sukochev for indicating the following reference to us (\cite{Kal86}). It seems that similar techniques can be used to construct a left-$1$-monotone space which is not $p$-convex for any $p$.
\end{remark}

\section{Application to noncommutative Khintchine inequalities} \label{section:khintchine}

Let $\A$ be a noncommutative probability space, $(\xi_i)_{i\in\Nb}$ a sequence of elements of $A$, orthonormal in $L_2(\A)$ and $\M$ a noncommutative measure space. The notation $\B(\ell^2) \overline{\otimes} L_{\infty}(0,1)$ will stand for the tracial tensor product of $\B(\ell^2)$ and $L_{\infty}(0,1)$ and is therefore equipped with a natural trace $\tau$, making it a noncommutative measure space. Similarly, define $\M_\A = \M\overline{\otimes}\A$. We keep using notations defined in the introduction. In this section, we impose that $\alpha = \infty$. In other words, symmetric spaces $E$ will be considered to be defined over $(0,\infty)$. 

We start with the key proposition of this section, an application of the Shcur-Horn theorem.

\begin{property}\label{prop:useSchurHorn}
Let $\M = \B(\ell^2) \overline{\otimes} L_{\infty}(0,1)$. Let $f,g\in F$ non-increasing such that $f^2 \succ g^2$ and $\Norm{f}_2 = \Norm{g}_2$ then there exists $x\in S(\M)$ such that $\mu(Gx) = f$ and $\mu(S_r(x)) = g$.
\end{property}

\begin{proof}
Write $f = \Sum{i=1}{N}a_i\Ind_{[(i-1)2^{-n},i2^{-n})}$ and $g = \Sum{i=1}{N}b_i\Ind_{[(i-1)2^{-n},i2^{-n})}$. The Schur-Horn theorem applied to $(a_i^2)$ and $(b_i^2)$ produces a positive matrix $M\in\Mb_N(\Cb)$ such that the eigenvalues of $M$ are given by $(a_i^2)$ and the diagonal of $M$ is given by $(b_i^2)$. Consider $x = (e_{i,i}M^{1/2} \otimes \Ind_{(0,2^{-n})})_{0\leq i \leq N}$. Then, 
\begin{align*}
\md{Gx}^2 =& \Sum{i,j}{} M^{1/2}e_{i,i}e_{j,j}M^{1/2} \otimes \Ind_{(0,2^{-n})} \otimes \xi_i^*\xi_j \\
=& \Sum{i}{} M^{1/2}e_{i,i}M^{1/2}\otimes \Ind_{(0,2^{-n})} \otimes 1 = M \otimes \Ind_{(0,2^{-n})} \otimes 1.
\end{align*}
Similar computations give: $S_c(x)^2 = M \otimes \Ind_{(0,2^{-n})}$ and $S_r(x)^2 = Diag(M) \otimes \Ind_{(0,2^{-n})}$.
\end{proof}

\begin{lemma}\label{lem:Kh+- implies weak p-monotonicity}
Let $E$ be a symmetric function space and $\M = \B(\ell^2) \overline{\otimes} L_{\infty}(0,1)$. Then,
$$Kh_\cap(E,\M) \Rightarrow \p{\forall f,g\in F, f^2 \succ g^2 \Rightarrow \Norm{g}_E \les \Norm{f}_E},$$
and
$$Kh_\Sigma(E,\M) \Rightarrow \p{\forall f,g\in F, f^2 \triangleright g^2 \Rightarrow \Norm{g}_E \les \Norm{f}_E}.$$
\end{lemma}

\begin{proof}
Suppose $Kh_\cap(E,\M)$. Let $f,g\in F$ such that $f^2 \succ g^2$. We can find $h\geq g\in F$ such that $f^2 \succ h^2$ and $\Norm{f^2} = \Norm{h^2}$ by Lemma \ref{lem:technique:fmaj g with norm f = norm g}. Then by Lemma \ref{prop:useSchurHorn} there exists $x\in S(\M)$ such that $\mu(Gx) = f$ and $\mu(S_r(x)) = h$. Hence, by $Kh_+(E,\M)$,
$$\Norm{f}_E = \Norm{Gx}_{E} \gtrsim \Norm{S_r(x)}_E = \Norm{h}_E.$$

Suppose $Kh_\Sigma(E,\M)$. Let $f,g\in F$ such that $f^2 \triangleright g^2$. We can find $h \geq g \in F$ such that $f^2 \triangleright h^2$ and $\Norm{h}_2 = \Norm{f}_2$ again by Lemma \ref{lem:technique:fmaj g with norm f = norm g}.
So by Lemma \ref{prop:useSchurHorn}, there exists $x\in S(\M)$ such that $\mu(Gx) = h$ and $\mu(S_r(x)) = f$. By $Kh_\Sigma(E,\M)$, $\Norm{Gx}_E \les \Norm{x}_{R_E + C_E}$. In particular, $\Norm{Gx}_E \les \Norm{S_r(x)}_E$. Hence, $\Norm{h}_E \les \Norm{f}_E$. Recall that $g\leq h$ so $\Norm{g}_E \leq \Norm{h}_E$ and the proof is complete.
\end{proof}

\begin{lemma}\label{lem:p-montokh}
Let $E$ be a symmetric function space and $x\in S(\M)$.
\begin{enumerate}
\item If $E$ is left-$2$-monotone then $\Norm{Gx}_E \gtrsim \Norm{x}_{R_E \cap C_E}$. 
\item If $E$ is right-$2$-monotone then $\Norm{Gx}_E \les \Norm{x}_{R_E + C_E}$.
\end{enumerate}
\end{lemma}

\begin{proof}
Consider the conditional expectation $\Eb : \M_{\A} \to \M$ defined by $\Eb := Id \otimes \tau_{\A}$. Note that for any $x\in S(\M)$,
$$\Eb(\md{Gx}^2) = S_c(x)^2\ \text{and}\ \Eb(\md{(Gx)^*}^2) = S_r(x)^2.$$
Hence, $\md{Gx}^2 \succ S_c(x)^2, S_r(x)^2$ and $\md{Gx}^2 \triangleleft S_c(x)^2, S_r(x)^2$ which is exactly what we needed.
\end{proof}

First, we state a characterisation true if the $(\xi_i)$ is a sequence of free unitaries. This is the simplest case since these variables are well behaved at $L_{\infty}$.

\begin{theorem}\label{thm:khfree}
Let $E$ be a quasi-Banach symmetric function space with the Fatou property. Assume that $(\xi_i)_{i\in\Nb}$ is a sequence of free Haar unitaries and that $\B(\ell^2) \overline{\otimes} L_{\infty}(0,1)$ embeds in $\M$. Then,
\begin{enumerate}
\item $Kh_\cap(E,\M) \Leftrightarrow E\in Int(L_2,L_\infty),$
\item $Kh_\Sigma(E,\M) \Leftrightarrow \exists p\in (0,2), E\in Int(L_p,L_2).$
\end{enumerate}
\end{theorem}

\begin{proof}(1),"$\Rightarrow$". Assume that $Kh_\cap(E,\M)$ holds. By Lemma \ref{lem:Kh+- implies weak p-monotonicity} and \ref{lem:FtoE1}, $E$ is left-$2$-monotone. Hence by \ref{prop:KmonToInter}, $E$ is an interpolation space for the couple $(L_2,L_\infty)$.

(1),"$\Leftarrow$". By Theorem \ref{thm:Kmon for Lp literature}, $E$ is left-$2$-monotone so by Lemma \ref{lem:p-montokh} one inequality is verified. The other inequality is true for any symmetric space by Lemma 2.1. in \cite{DirRic13}.

(2),"$\Rightarrow$". By Lemma \ref{lem:Kh+- implies weak p-monotonicity} and Lemma \ref{lem:concave}, $E$ is right-$2$-monotone. Furthermore, by corollary \ref{cor:allpmon}, $E$ is also left-$p$-monotone for some $p$ and we can assume that $p<2$. Hence by Theorem \ref{thm:functionspace}, $E\in Int(L_p,L_2)$.

(2),"$\Leftarrow$". By Theorem \ref{thm:functionspace}, $E$ is right-$2$-monotone so by Lemma \ref{lem:p-montokh}, one inequality is verified. The other holds without any assumption on $E$. Indeed, by Remark \ref{rem:beta} and Proposition \ref{prop:boydintinfty}, there exists $p>0$ such that $E \in Int(L_p,L_\infty)$ ($p$ can be assumed to be less than $2$). Note also that for the chosen variables, the Khintchine inequalities hold in $L_p$ by \cite{PisRic17} (or \cite{Cad18}). This allows to conclude with \cite{Cad18}, Theorem 2.6.
 
\end{proof}

Now we turn our attention to Rademacher variables. We start with an lemma due to Astashkin in \cite{Ast09} (Theorem 7.2), we simply verify it in the semifinite setting.

\begin{lemma}\label{lem:kh_alpha}
Assume that $E$ has the Fatou property, that $(\xi_i)_{i\in\Nb}$ is a sequence of independent Rademacher variables and that $Kh_{\cap}(E,L_{\infty}(0,\alpha))$ holds. Then, $\alpha_E \neq 0$.
\end{lemma}

\begin{proof}
Let $f$ be a bounded and finitely supported function in $E$. Consider the constant sequence $x = \p{\frac{f}{\sqrt{n}}}_{1\leq i\leq n}$. Using the notation introduced in Lemma \ref{lem:alpha=0}, and applying $Kh_\cap(E,L_{\infty}(0,\alpha))$ to $x$, we obtain:
$$\Norm{T_{\varphi_n}f}_E \les \Norm{f}_E,\ \text{where}\ \varphi_n = \frac{1}{\sqrt{n}}\sum_{i=1}^n \xi_i.$$
By the Fatou property, this inequality extends to all $f$ in $E$ and the implicit constant does not depend
on $f$ nor on $n$ since it is given by $Kh_\cap(E,L_{\infty}(0,\alpha))$.
In other words, the operators $T_{\varphi_n}$ are uniformly bounded on $E$. 
Furthermore, by the central limit theorem, the sequence $(\varphi_n)_{n\in\Nb}$ converges in distribution to a gaussian variable and is therefore not uniformly bounded in $L_\infty$. 
By Lemma \ref{lem:alpha=0}, this means that $\alpha_E \neq 0$.
\end{proof}

\begin{theorem}\label{thm:khrad}
Let $E$ be a quasi-Banach symmetric function space with the Fatou property. Assume that $(\xi_i)_{i\in\Nb}$ is a sequence of independent Rademacher variables and that $\M$ contains $\B(\ell^2) \overline{\otimes} L_{\infty}(0,1)$. Then,
\begin{enumerate}
\item $Kh_\cap(E,\M) \Leftrightarrow \exists q\in (2,\infty), E\in Int(L_2,L_q),$
\item $Kh_\Sigma(E,\M) \Leftrightarrow \exists p\in (0,2), E\in Int(L_p,L_2).$
\end{enumerate}
\end{theorem}

\begin{proof}(1),"$\Rightarrow$". Assume that $Kh_\cap(E,\M)$ holds. Again, by Lemma \ref{lem:Kh+- implies weak p-monotonicity} and \ref{lem:FtoE1}, $E$ is left-$2$-monotone. By Lemma \ref{lem:kh_alpha}, we also know that $\alpha_E \neq 0$. So by corollary \ref{cor:alphaqmon}, there exists $q<\infty$ such that $E$ is right-$q$-monotone. Hence, by Theorem \ref{thm:functionspace}, $E$ is an interpolation space for the couple $(L_2,L_q)$

(1),"$\Leftarrow$". By Theorem \ref{thm:Kmon for Lp literature}, $E$ is left-$2$-monotone so by Lemma \ref{lem:p-montokh} one inequality is verified. The other is true for in any $L_p$-space, $2\leq p<\infty$ by \cite{Lus86}. Hence, it holds in $E$ by interpolation.

(2) can be proven in a similar way than Theorem \ref{thm:khfree}, (2).
\end{proof}

\begin{remark}
By considering unbounded sequences of operators when defining the properties $Kh_\cap$ and $Kh_\Sigma$, theorems \ref{thm:khfree} and \ref{thm:khrad} could be formulated and proven without assuming that $E$ has the Fatou property. If $E$ is a sequence space (not necessarily Fatou) and $\M = B(\ell^2)$, using an infinite dimensional version of the Schur-Horn theorem (see \cite{KW10}), one can argue similarly as above. For a general function space $E$, more work seems to be required and would be outside of the scope of this paper.
\end{remark}

\begin{remark}
By replacing $\B(\ell^2) \overline{\otimes} L_\infty(0,1)$ by the hyperfinite factor $\R$, one can obtain, with the same proofs, similar statements for symmetric spaces defined over $(0,1)$ and by renormalising the trace of $\R$, on $(0,\alpha)$, $\alpha <\infty$.
\end{remark}

\subsection*{Acknowledgements}

I would like to thank my advisor \'Eric Ricard for fruitful discussions and his reading of previous versions of the text. I am most grateful to M. Cwikel, F. Sukochev and D. Zanin for valuable exchanges leading to a substantial improvement of the paper.
\bibliographystyle{plain}

\begin{thebibliography}{}

\bibitem[Astashkin and Maligranda, 2004]{AM04}
Astashkin, S. and Maligranda, L. (2004).
\newblock Interpolation between {$L_1$} and {$L_p,\ 1<p<\infty$}.
\newblock {\em Proc. Amer. Math. Soc.}, 132(10):2929--2938.

\bibitem[Astashkin, 2009]{Ast09}
Astashkin, S.~V. (2009).
\newblock Rademacher functions in symmetric spaces.
\newblock {\em Sovrem. Mat. Fundam. Napravl.}, 32:3--161.

\bibitem[Bergh and L\"{o}fstr\"{o}m, 1976]{BerLof76}
Bergh, J. and L\"{o}fstr\"{o}m, J. (1976).
\newblock {\em Interpolation spaces. {A}n introduction}.
\newblock Springer-Verlag, Berlin-New York.
\newblock Grundlehren der Mathematischen Wissenschaften, No. 223.

\bibitem[Boyd, 1969]{Boy69}
Boyd, D.~W. (1969).
\newblock Indices of function spaces and their relationship to interpolation.
\newblock {\em Canadian J. Math.}, 21:1245--1254.

\bibitem[Cadilhac, 2019]{Cad18}
Cadilhac, L. (2019).
\newblock Noncommutative {K}hintchine inequalities in interpolation spaces of
  {$L_p$}-spaces.
\newblock {\em Adv. Math.}, 352:265--296.

\bibitem[Cwikel, 1981]{Cwi81}
Cwikel, M. (1981).
\newblock Monotonicity properties of interpolation spaces. {II}.
\newblock {\em Ark. Mat.}, 19(1):123--136.

\bibitem[Cwikel and Nilsson, 2018]{CwiNil18}
Cwikel, M. and Nilsson, P. (2018).
\newblock An alternative characterization of normed interpolation spaces
  between $\ell^1$ and $\ell^q$.
\newblock {\em preprint}.

\bibitem[Dirksen, 2011]{Dir11}
Dirksen, S. (2011).
\newblock {\em Noncommutative and Vector-valued Rosenthal Inequalities}.
\newblock PhD thesis, Technische Universiteit Delft.

\bibitem[Dirksen et~al., 2011]{DPPS11}
Dirksen, S., de~Pagter, B., Potapov, D., and Sukochev, F. (2011).
\newblock Rosenthal inequalities in noncommutative symmetric spaces.
\newblock {\em J. Funct. Anal.}, 261(10):2890--2925.

\bibitem[Dirksen and Ricard, 2013]{DirRic13}
Dirksen, S. and Ricard, E. (2013).
\newblock Some remarks on noncommutative {K}hintchine inequalities.
\newblock {\em Bull. Lond. Math. Soc.}, 45(3):618--624.

\bibitem[Hiai, 1987]{Hia87}
Hiai, F. (1987).
\newblock Majorization and stochastic maps in von {N}eumann algebras.
\newblock {\em J. Math. Anal. Appl.}, 127(1):18--48.

\bibitem[Holmstedt, 1970]{Hol70}
Holmstedt, T. (1970).
\newblock Interpolation of quasi-normed spaces.
\newblock {\em Math. Scand.}, 26:177--199.

\bibitem[Horn, 1954]{Hor54}
Horn, A. (1954).
\newblock Doubly stochastic matrices and the diagonal of a rotation matrix.
\newblock {\em Amer. J. Math.}, 76:620--630.

\bibitem[Kaftal and Weiss, 2010]{KW10}
Kaftal, V. and Weiss, G. (2010).
\newblock An infinite dimensional {S}chur-{H}orn theorem and majorization
  theory.
\newblock {\em J. Funct. Anal.}, 259(12):3115--3162.

\bibitem[Kalton, 1986]{Kal86}
Kalton, N.~J. (1986).
\newblock Banach envelopes of nonlocally convex spaces.
\newblock {\em Canad. J. Math.}, 38(1):65--86.

\bibitem[Kalton and Montgomery-Smith, 2003]{KalMon03}
Kalton, N.~J. and Montgomery-Smith, S.~J. (2003).
\newblock Interpolation of {B}anach spaces.
\newblock In {\em Handbook of the geometry of {B}anach spaces, {V}ol. 2}, pages
  1131--1175. North-Holland, Amsterdam.

\bibitem[Kalton and Sukochev, 2008]{KalSuk08}
Kalton, N.~J. and Sukochev, F. (2008).
\newblock Symmetric norms and spaces of operators.
\newblock {\em J. Reine Angew. Math.}, 621:81--121.

\bibitem[Kre\u{\i}n et~al., 1982]{KrePetSem82}
Kre\u{\i}n, S.~G., Petun\={i}n, Y.~I., and Sem\"{e}nov, E.~M. (1982).
\newblock {\em Interpolation of linear operators}, volume~54 of {\em
  Translations of Mathematical Monographs}.
\newblock American Mathematical Society, Providence, R.I.
\newblock Translated from the Russian by J. Sz\H{u}cs.

\bibitem[Le~Merdy and Sukochev, 2008]{LeMSuk08}
Le~Merdy, C. and Sukochev, F. (2008).
\newblock Rademacher averages on noncommutative symmetric spaces.
\newblock {\em J. Funct. Anal.}, 255(12):3329--3355.

\bibitem[Levitina et~al., 2020]{LevSukZan18}
Levitina, G., Sukochev, F., and Zanin, D. (2020).
\newblock Cwikel estimates revisited.
\newblock {\em Proc. Lond. Math. Soc. (3)}, 120(2):265--304.

\bibitem[Lorentz and Shimogaki, 1971]{LorShi71}
Lorentz, G.~G. and Shimogaki, T. (1971).
\newblock Interpolation theorems for the pairs of spaces {$(L^{p},\,L^{\infty
  })$} and {$(L^{1},\,L^{q})$}.
\newblock {\em Trans. Amer. Math. Soc.}, 159:207--221.

\bibitem[Lust-Piquard, 1986]{Lus86}
Lust-Piquard, F. (1986).
\newblock In\'{e}galit\'{e}s de {K}hintchine dans {$C_p\;(1<p<\infty)$}.
\newblock {\em C. R. Acad. Sci. Paris S\'{e}r. I Math.}, 303(7):289--292.

\bibitem[Lust-Piquard and Pisier, 1991]{LusPis91}
Lust-Piquard, F. and Pisier, G. (1991).
\newblock Noncommutative {K}hintchine and {P}aley inequalities.
\newblock {\em Ark. Mat.}, 29(2):241--260.

\bibitem[Lust-Piquard and Xu, 2007]{LusXu07}
Lust-Piquard, F. and Xu, Q. (2007).
\newblock The little {G}rothendieck theorem and {K}hintchine inequalities for
  symmetric spaces of measurable operators.
\newblock {\em J. Funct. Anal.}, 244(2):488--503.

\bibitem[Montgomery-Smith, 1996]{Mon96}
Montgomery-Smith, S.~J. (1996).
\newblock The {H}ardy operator and {B}oyd indices.
\newblock In {\em Interaction between functional analysis, harmonic analysis,
  and probability ({C}olumbia, {MO}, 1994)}, volume 175 of {\em Lecture Notes
  in Pure and Appl. Math.}, pages 359--364. Dekker, New York.

\bibitem[Pisier and Ricard, 2017]{PisRic17}
Pisier, G. and Ricard, E. (2017).
\newblock The non-commutative {K}hintchine inequalities for {$0<p<1$}.
\newblock {\em J. Inst. Math. Jussieu}, 16(5):1103--1123.

\bibitem[Sparr, 1978]{Spa78}
Sparr, G. (1978).
\newblock Interpolation of weighted {$L_{p}$}-spaces.
\newblock {\em Studia Math.}, 62(3):229--271.

\end{thebibliography}

\

\emph{Laboratoire de mathématiques Nicolas Oresme, Université de Caen Normandie, 14032 Caen Cedex, France.}

\

 \emph{Mathematical Institute of the Polish Academy of Sciences,
  ul. \'Sniadeckich 8
00-656 Warszawa, Poland}

\

\emph{E-mail address}: lcadilhac@impan.pl
\end{document}